\documentclass[12pt,a4paper]{extarticle}

\newif\ifdraft 
\draftfalse    

\usepackage{latexsym}
\usepackage{amssymb}
\usepackage{amsmath}
\usepackage{amsthm}
\DeclareFontFamily{U}{matha}{\hyphenchar\font45}
\DeclareFontShape{U}{matha}{m}{n}{
      <5> <6> <7> <8> <9> <10> gen * matha
      <10.95> matha10 <12> <14.4> <17.28> <20.74> <24.88> matha12
      }{}
\DeclareSymbolFont{matha}{U}{matha}{m}{n}
\DeclareFontSubstitution{U}{matha}{m}{n}
\DeclareMathSymbol{\hashsym}{\mathbin}{matha}{'043}
\usepackage{thmtools}
\usepackage{enumitem}
\usepackage{cite}
\usepackage[all]{xy}
\usepackage{framed}
\usepackage{graphicx}
\usepackage{parskip}
\usepackage{hyperref}
\usepackage{bm}
\usepackage[a4paper,margin=2cm]{geometry}
\usepackage[utf8]{inputenc}
\usepackage{mathdots}
\usepackage{pifont}
\usepackage{verbatim}
\usepackage{xcolor}
\usepackage{marginnote}

\begingroup
    \makeatletter
    \@for\theoremstyle:=definition,remark,plain\do{%
        \expandafter\g@addto@macro\csname th@\theoremstyle\endcsname{%
            \addtolength\thm@preskip\parskip
            }%
        }
\endgroup

\makeatletter
\xydef@\txt@ii#1{\vbox{\vspace*{-5pt}%
 \let\\=\cr
 \tabskip=\z@skip \halign{\relax\hfil\txtline@@{##}\hfil\cr\leavevmode#1\crcr}}}
\makeatother

\theoremstyle{definition}
\newtheorem{thm}{Theorem}[section]
\newtheorem{lem}[thm]{Lemma}
\newtheorem{cor}[thm]{Corollary}
\newtheorem{defn}[thm]{Definition}
\newtheorem{propn}[thm]{Proposition}
\newtheorem*{thm*}{Theorem}
\newtheorem{notn}[thm]{Notation}

\newtheorem*{qn*}{Question}
\newtheorem*{conj}{Conjecture}
\newtheorem{props}[thm]{Properties}

\newtheorem*{nts}{Note to self}

\theoremstyle{remark}
\newtheorem{rk}[thm]{Remark}
\newtheorem{rks}[thm]{Remarks}

\newtheoremstyle{custthm}{\parskip}{}{\normalfont}{}{\bfseries}{.}{ }{\thmname{#1} \thmnote{#3}}
\theoremstyle{custthm}

\newcommand{\Aut}{\mathrm{Aut}}
\newcommand{\Inn}{\mathrm{Inn}}

\newcommand{\Spec}{\mathrm{Spec}}

\newcommand{\invlim}{\underset{\longleftarrow}{\lim}}

\newcommand{\C}{\mathbf{C}}
\newcommand{\pr}{\mathrm{pr}}
\newcommand{\hash}{\mathbb{\hashsym}}

\newcommand\checkedbox{%
 \makebox[1.4em][c]{%
 \makebox[0pt][c]{\raisebox{.1em}{\small{\textbf{\ding{52}}}}}%
 \makebox[0pt][c]{\Large$\square$}}%
 }%

\newcommand{\checked}{%
 \ifdraft
 \hspace{-20pt}\checkedbox\vspace{-25pt}%
 \fi
 }%

\ifdraft
 {%
 \end{framed}%
 }%
\fi

{%
\begin{framed}\begin{nts}%
}%
{%
\end{nts}\end{framed}%
}

\newif\ifdark

\makeatletter
\newcommand{\globalcolor}[1]{%
  \color{#1}\global\let\default@color\current@color
}
\makeatother

\darkfalse

\ifdark
\pagecolor{black}
\fi

\ifdark
\AtBeginDocument{\globalcolor{white}}
\else
\AtBeginDocument{\globalcolor{black}}
\fi

\begin{document}

\numberwithin{equation}{section}
\binoppenalty=\maxdimen
\relpenalty=\maxdimen

\title{The conjugation action in completed group rings}
\author{William Woods\footnote{University of Essex, UK. Email address: \texttt{billywoods@gmail.com}.}}
\date{\today}
\maketitle
\begin{abstract}
Let $k = \mathbb{F}_p$ or $\mathbb{Z}_p$ (or finite extensions of these). Let $G$ be a $p$-valuable group, and form its completed group algebra $kG$. By analysing the conjugation action of $G$ on itself, we prove two structural results. Firstly, we show that all inner automorphisms of $kG$ that preserve $G$ are induced from inner automorphisms of $G$. Secondly, for a closed subgroup $\Gamma$ of $G$, we calculate the $\Gamma$-fixed ring of $kG/I$ under the conjugation action of $\Gamma$, for certain ideals $I$ induced from the $G$-centraliser of $\Gamma$.
\end{abstract}

\tableofcontents

%
%
%
%
%
%

\newpage

\section*{Introduction}

Throughout this paper, $p$ will denote a fixed prime number, and $\mathbb{Z}_p$ will be the ring (or additive group) of $p$-adic integers.

Our main objects of study are the completed group rings $kG$ of \emph{compact $p$-adic analytic groups} $G$ over certain profinite rings $k$ (usually $\mathbb{F}_p$ or $\mathbb{Z}_p$). The rings $kG$ are often known as \emph{Iwasawa algebras} in the literature, and play an important role in modern number theory.

Compact $p$-adic analytic groups (also known as compact $p$-adic Lie groups) may be most simply defined as closed subgroups of the topological group $GL_n(\mathbb{Z}_p)$: see e.g. \cite{DDMS, lazard} for references. In studying these groups, it is common first to study a smaller subclass consisting of particularly nice groups: common choices include \emph{uniform} \cite[\S 4]{DDMS} groups, or the slightly more general \emph{$p$-valuable} \cite[III, 3.1.6--7]{lazard} groups.

Both $p$-valuable and uniform groups are torsion-free pro-$p$ groups of finite rank with well-behaved filtrations, which makes them easier to study. Moreover, arbitrary compact $p$-adic analytic groups are ``almost" $p$-valuable (or uniform) in several important senses: for instance, given a compact $p$-adic analytic group $G$, we can always find an open (i.e. closed and finite-index) subgroup $G_0$ which is both $p$-valuable (or uniform) and normal in $G$ \cite[Corollary 8.34]{DDMS}. Having chosen such a $G_0$, we can form the ring $kG$ as a crossed product of the ring $kG_0$ with the finite group $G/G_0$: by doing this, we can often pass information from $kG_0$ to $kG$. Such descriptions have been concretely applied with fruitful results in previous work: see, for instance, \cite{woods-extensions-of-primes,woods-catenary, woods-prime-quotients, ardakovbrown}.

For this reason, in the current paper, we work primarily with $p$-valuable and uniform groups $G$. The most important properties of these groups are recalled for the convenience of the reader below in \S \ref{subsec: p-val groups} and \S \ref{subsec: uniform groups}.

There is an ongoing project to understand the (prime) ideals in these completed group rings $kG$: various precise questions were posed in the survey paper \cite[Questions G--O]{ardakovbrown}, and despite significant work on this front in the last decade, most of these questions still remain partly or entirely open. The most substantial results in this direction have been those of the paper \cite{ardakovInv}, which answers special cases of \cite[Questions N and O]{ardakovbrown}:

\hfill\begin{minipage}{\dimexpr\textwidth-1cm}
\emph{Let $G$ be a nilpotent uniform group with centre $Z$.
\begin{itemize}
\item Is every prime ideal of $\mathbb{F}_p G$ completely prime?
\item Is $\Spec(\mathbb{F}_p G)$ the disjoint union of finitely many commutative strata?
\item Is every faithful prime ideal of $\mathbb{F}_p G$ controlled by $Z$?
\end{itemize}
}
\end{minipage}

More recently, the results of the thesis \cite{jones-thesis} achieve positive results towards the more general case of \emph{soluble} uniform $G$.

Many of the questions in \cite{ardakovbrown} were posed exclusively in the (``\emph{modular}") case when $k = \mathbb{F}_p$. The (``\emph{integral}") case $k = \mathbb{Z}_p$ is often of even greater interest in areas of number theory, where Iwasawa algebras originated; and the study of the ideals $I$ of $\mathbb{F}_pG$ may be viewed as the study of the ideals $J$ of $\mathbb{Z}_pG$ which contain $p$, so that the modular case is a special case of the integral case. However, the integral case appears to be considerably more difficult, and currently very little is known about ideals of $\mathbb{Z}_pG$ which do not contain $p$. For some recent developments, see \cite[Theorems C and E]{jones-thesis} and \cite{jones-woods-1}.

In all of the below results, $k$ will denote \emph{either} a finite field of characteristic $p$ \emph{or} a complete discrete valuation ring whose residue field has characteristic $p$.

\subsection{Inner automorphisms}

Our first results concern automorphisms of $kG$.

\textbf{Theorem A.} Let $G$ be a $p$-valuable group, and let $\sigma$ be an automorphism of $kG$ that preserves the naturally embedded copy of $G$. Then $\sigma$ is the (unique) extension of an inner automorphism of $G$.\qed

The importance of this result stems from an analysis of the behaviour of skew power series rings $R[[x;\sigma,\delta]]$ (see \cite{schneider-venjakob-codim,letzter-noeth-skew,woods-SPS-dim} for the definition and basic properties of these rings in the Iwasawa context). Forthcoming joint work of the current author and Adam Jones shows that the behaviour of $R[[x;\sigma,\delta]]$ varies depending on whether $\sigma$ is inner or outer. A similar dichotomy has also been observed in the world of skew \emph{polynomial} rings $R[x;\sigma,\delta]$, e.g. \cite[Theorem 3.6]{LerMat92}.

If $G$ is a $p$-valuable group and $H$ is a closed normal subgroup such that $G/H\cong \mathbb{Z}_p$, choose an arbitrary topological generator $gH$ for $\mathbb{Z}_p$: that is, choose an element $g\in G$ such that $G = \overline{\langle H, g\rangle}$. Then $kG$ can be written as a skew power series ring $kH[[x;\sigma,\sigma-\mathrm{id}]]$, where $x = g - 1$ and $\sigma$ is conjugation by $g$ \cite[\S 4]{schneider-venjakob-codim}. Hence $\sigma$ is an automorphism of $kH$ which preserves the naturally embedded copy of $H$, and so if it is inner, we can apply Theorem A to obtain the following simplification.

\textbf{Corollary B.} Let $G$ be a $p$-valuable group and $H$ a closed normal subgroup such that $G/H\cong \mathbb{Z}_p$. Choose an element $g\in G$ such that $G = \overline{\langle H, g\rangle}$. If the conjugation action of $g$ on $kH$ is inner, then $G \cong H\times \mathbb{Z}_p$.\qed

\subsection{Fixed rings}

Next, we turn our attention to subrings of $kG$ fixed under the conjugation action of a closed subgroup $\Gamma\leq G$.

\textbf{Theorem C.}
Let $G$ be a $p$-valuable group, $\Gamma$ an arbitrary closed subgroup which we consider acting on $G$ by conjugation, and $H$ the centraliser in $G$ of $\Gamma$. If $I$ is a faithful $G$-prime ideal of $kH$, then the set of $\Gamma$-invariant elements of $kG/IkG$ is precisely $(kH+IkG)/IkG$.\qed

Note that, under this setup, $kG/IkG$ is a ring, as $H$ is a \emph{normal} subgroup: see Lemma \ref{lem: centralisers are normal} below. Note also that we may identify the subring $(kH + IkG)/IkG \leq kG/IkG$ with the ring $kH/I$. (Indeed, $kG$ is a \emph{flat} $kH$-module by \cite[Lemma 4.5]{brumer}, and is \emph{faithfully} flat by \cite[Proposition 7.2.3]{MR} as the maximal ideal of $kG$ contains the maximal ideal of $kH$. Now it follows that $IkG\cap kH = I$, as in \cite[Lemma 5.1]{ardakovInv}.)

This result, proved below as Theorem \ref{thm: fixed ring over F_p}, is an analogue of Roseblade's \cite[Lemma 10]{roseblade}, which was used to answer similar questions about group algebras of polycyclic groups. While Theorem A is not quite as general as \cite[Lemma 10]{roseblade}, in practice it is general enough to be used in an analogous context:

\textbf{Corollary D.}
Let $G$ be a supersoluble $p$-valuable group, and $Z$ its centre. Then $kG$ is centrally separated if and only if all its faithful prime ideals are controlled by $Z$.\qed

Here, by \emph{centrally separated}, we mean the following: if $P\lneq Q$ are two elements of $\Spec(kG)$, then $Q/P$ contains a nonzero \emph{central} element of $kG/P$. An ideal $I\lhd kG$ is \emph{faithful} if $G$ embeds naturally into $kG/I$, i.e. if the kernel of the natural map $G\to (kG)^\times \to (kG/I)^\times$ is trivial.

This result is proved in \S \ref{subsec: ideal separation}. Of course, the question of when $kG$ satisfies either condition remains open in general: to our knowledge, the only case that is known is that of \cite{ardakovInv}. Still, we explore the potential significance of this result in the larger context, and give some conjectures.

\subsection{Ideals and normal subgroups}

Finally, in \S \ref{subsec: parallels}, we provide a first result linking the ideal structure of $\mathbb{Z}_pG$ with the normal subgroup structure of $G$:

\textbf{Theorem E.}
Let $G$ be a soluble $p$-valuable group. If $\mathbb{Z}_p G$ is polycentral, than $G$ is in fact nilpotent. \qed

This mirrors similar results for group algebras of polycyclic-by-finite groups, as well as finite-dimensional Lie algebras: see \S \ref{subsec: parallels} for a discussion.

\section{Preliminaries}

\subsection{$p$-valuable groups}\label{subsec: p-val groups}

As mentioned in the introduction, every compact $p$-adic analytic group $G$ contains an open $p$-valuable normal subgroup $G_0$, and $p$-valuable groups are in some sense much better behaved. We recall some of the basic properties of $p$-valuable groups.

\begin{defn}
If $H$ is a closed subgroup of $G$, then $H$ is called \emph{orbital} if it has finitely many $G$-conjugates. An orbital closed subgroup is called \emph{isolated} if, whenever $H'$ is an orbital closed subgroup of $G$ strictly containing $H$, we have $[H':H]=\infty$. \cite[Definition 1.4]{woods-struct-of-G}.
\end{defn}

\begin{props}\label{props: p-valuable groups}
Let $G$ be a $p$-valuable group, as defined in \cite[III, 3.1.6]{lazard}.

\begin{enumerate}[label=(\roman*)]
\item $G$ is a torsion-free pro-$p$ group of finite rank. \cite[III, 2.1.2]{lazard}
\item $G$ is \emph{orbitally sound} \cite[Proposition 5.9]{ardakovInv}, i.e. if $H$ is any orbital closed subgroup of $G$, then its normal core $H^\circ := \bigcap_{g\in G} H^g$ has finite index in $H$. In particular, this implies that its isolated orbital closed subgroups are normal in $G$ \cite[Lemma 1.10]{woods-struct-of-G}. Hence an orbital closed subgroup $H$ of $G$ is isolated if and only if it is normal and $G/H$ is torsion-free (compare \cite[\S 5.3]{ardakovInv}), and by (i), this means that an orbital closed subgroup $H$ is isolated if and only if it is normal and $x\in G$, $x^p\in H\implies x\in H$.
\item Closed subgroups of $G$ are $p$-valuable \cite[III, 2.1.2]{lazard}.
\item If $H$ is an isolated normal closed subgroup of $G$, then $G/H$ is $p$-valuable. \cite[IV, 3.4.2]{lazard}
\item If $x,y\in G$ satisfy $x^p = y^p$, then $x = y$. \cite[III, 2.1.4]{lazard}
\end{enumerate}
\end{props}

\begin{lem}\label{lem: centralisers are normal}\label{lem: G/H is p-val}
Let $G$ be a $p$-valuable group, $\Gamma$ an arbitrary normal closed subgroup which we consider acting on $G$ (on the right) by conjugation, and $H$ the centraliser of $\Gamma$ in $G$. Then $H$ is an isolated normal closed subgroup of $G$, and $G/H$ is $p$-valuable.
\end{lem}

\begin{proof}
Given a sequence $h_n\in H$ converging to $h\in G$, we must have $\gamma = h_n^{-1}\gamma h_n\to h^{-1}\gamma h$ for all $\gamma\in \Gamma$, showing that $h$ centralises $\Gamma$. Hence $h\in H$, and so as the convergent sequence $(h_n)$ was arbitrary, $H$ must be closed in $G$. The normality of $H$ follows from the normality of $\Gamma$.

To show that $H$ is isolated using the condition given in Property \ref{props: p-valuable groups}(ii), we need only show that, if $h^p\in H$ for some $h\in G$, then in fact $h\in H$. So let $\gamma\in \Gamma$ be arbitrary: then we know that $\gamma^{-1}h^p\gamma = h^p$, i.e. $(\gamma^{-1}h\gamma)^p = h^p$, and so by Property \ref{props: p-valuable groups}(v) we get $\gamma^{-1}h\gamma = h$, i.e. $h$ centralises $\Gamma$. So $H$ is isolated, and $G/H$ is $p$-valuable by Property \ref{props: p-valuable groups}(iv).
\end{proof}

\subsection{Profinite spaces, groups and rings}

Our groups and rings are all profinite, so we recall the basics of inverse limits necessary for our purposes: cf. \cite[\S 1]{DDMS}, \cite[\S 1]{wilson}.

\begin{defn}
Let $(I,\leq)$ be a directed set, i.e. a nonempty partially ordered set with the property that, for all $i,j\in I$, there exists $k\in I$ such that $k \geq i$ and $k \geq j$.

Let $\mathcal{C}$ be the category of topological spaces, topological groups or topological rings. An \emph{inverse system} in $\mathcal{C}$ (over $I$) is a collection of objects $A_i$ for all $i\in I$, together with a collection of morphisms $\pi_{ij}: A_i \to A_j$ for all $i,j\in I$ satisfying $i \geq j$, with the property that $\pi_{ik} = \pi_{jk}\circ \pi_{ij}$ for all $i\geq j\geq k$: we will usually write this inverse system simply as $(A_i)_{i\in I}$, omitting the $\pi_{ij}$ if they are known from context.

The \emph{inverse limit} of the inverse system $(A_i)_{i\in I}$ is denoted $\invlim_{i\in I} A_i$, and is defined to be the object
$$\underset{i\in I}{\invlim}\, A_i := \left\{(a_i)_{i\in I} \in \prod_{i\in I} A_i : \pi_{jk}(a_j) = a_k \text{ for all } j \geq k\right\}.$$
This comes equipped with natural projection maps $\varphi_j: \invlim_{i\in I} A_i \to A_j$, and is given the coarsest topology making all the projection maps $\varphi_j$ continuous, making it again an object of $\mathcal{C}$.

Finally, if the $A_i\in\mathcal{C}$ are finite and discrete, then $\invlim_{i\in I} A_i\in\mathcal{C}$ is called \emph{profinite}.
\end{defn}

\begin{props}\label{props: profinite objects}
$ $

\begin{enumerate}[label=(\roman*)]
\item Let $(A_i)_{i\in I}$ be an inverse system of finite discrete objects in $\mathcal{C}$. Then $\invlim_{i\in I} A_i$ is a closed subobject of $\prod_{i\in I} A_i$ (cf. \cite[proof of Proposition 1.3]{DDMS}).
\item Suppose that $J\subseteq I$ is a directed set. Then there is a canonical morphism $\invlim_{i\in I} A_i \to \invlim_{j\in J} A_j$. Moreover, if $J$ is \emph{cofinal} with $I$, i.e. for all $i\in I$ there exists $j\in J$ such that $j\geq i$, then this canonical map is an isomorphism in $\mathcal{C}$.
\item If $G$ is a profinite group, then $G\cong \invlim_{N} (G/N)$, where $N$ ranges over the inverse system of open normal subgroups of $G$. The partial order of this directed set is reverse containment, i.e. $N_2 \geq N_1$ if $N_2 \subseteq N_1$, and the morphisms $G/N_2 \to G/N_1$ of this inverse system are the natural projection homomorphisms.
\end{enumerate}
\end{props}

\subsection{Topological rings}

Let $R$ be a ring equipped with a topology $\mathcal{T}$. We say that $\mathcal{T}$ is a \emph{ring topology} on $R$ if the addition and multiplication maps $R\times R\to R$ and the negation map $R\to R$ are all continuous with respect to $\mathcal{T}$, and when this is the case, we say $R$ is a \emph{topological ring} (with respect to $\mathcal{T}$). All rings of interest in this paper are profinite (see Remark \ref{rk: kG is profinite} below), and hence carry the profinite topology \cite[\S 1]{wilson}.

Our standard reference for general topological rings is \cite{warner}. 

\checked\begin{lem}\label{lem: cpt t.d. rings}
Let $R$ be a compact, totally disconnected (topological) ring and $I$ a closed ideal. Then $R/I$ is a compact, totally disconnected ring.
\end{lem}

\begin{proof}
Clearly $R/I$ is a ring, and is compact as a quotient space. By \cite[Theorem 5.4]{warner}, the topology on $R/I$ is indeed a ring topology, so that $R/I$ is a compact ring. Now it follows from \cite[Theorem 5.17(3)]{warner} that $R/I$ is totally disconnected.
\end{proof}

\checked\begin{lem}\label{lem: a profinite ring is cpt t.d.}
Profinite rings are compact and totally disconnected.
\end{lem}

\begin{proof}
By definition, a profinite ring $R$ is the inverse limit of an inverse system of discrete finite rings, and is hence a closed subspace of the product space. But discrete finite rings are compact and totally disconnected, and these properties are preserved under arbitrary products and closed subspaces.
\end{proof}

\checked\begin{lem}\label{lem: inv lim of quotients}
Let $R$ be a compact, totally disconnected ring. Suppose that $\mathcal{J}$ is a fundamental system of neighbourhoods of zero for $R$ consisting of open ideals, and let $I$ be an arbitrary closed ideal of $R$. Then the canonical ring homomorphism
$$R/I \to \underset{J\in\mathcal{J}}\invlim \left(R/(J+I)\right)$$
is an isomorphism of topological rings.
\end{lem}

\begin{proof}
By Lemma \ref{lem: cpt t.d. rings}, $R/I$ is a compact, totally disconnected ring. As $\mathcal{J}$ is a fundamental system of neighbourhoods of zero in $R$, by \cite[Theorem 5.5]{warner}, we have that $\{A_J \;:\; J\in\mathcal{J}\}$ is a fundamental system of neighbourhoods of zero in $R/I$, where we set $A_J := (J + I) / I$. Each $A_J$ is easily seen to be an open ideal in $R$, as it is the image of $J$ under the projection map $R\to R/I$, which is an open map (see e.g. \cite[Theorem 5.2]{warner}).

Now \cite[Theorem 5.23]{warner} implies that $R/I$ is topologically isomorphic to the inverse limit of the inverse system of rings $\{(R/I)/A_J \;:\; J\in\mathcal{J}\}$. But $(R/I)/A_J$ is naturally isomorphic to $R/(J + I)$ by standard isomorphism theorems. (That this topological isomorphism is canonical can be read off from \cite[Corollary 5.22]{warner} and the proof of \cite[Theorem 5.23]{warner}.)
\end{proof}

\subsection{Completed group algebras}

Let $G$ be a compact $p$-adic analytic group, and let $k$ be \emph{either} a finite field of characteristic $p$ \emph{or} a complete discrete valuation ring whose residue field is a finite field of characteristic $p$.

\begin{defn}\label{defn: completed group ring}
The \emph{completed group ring} of $G$ over $k$ is defined to be
$$kG := \underset{U}{\invlim}\;  k[G/U],$$
where $U$ ranges over the inverse system of open normal subgroups of $G$, and $k[G/U]$ is the usual group algebra of the (finite) group $G/U$ over the commutative ring $k$. The topological rings $\mathbb{Z}_pG$ and $\mathbb{F}_pG$ are sometimes called the (respectively \emph{integral} and \emph{modular}) \emph{Iwasawa algebras} of $G$. When it is notationally convenient, we may write $k[[G]]$ instead of $kG$; when $G$ is finite, we may write $k[G]$, as the completed group ring and the usual group ring are canonically isomorphic in this case.

The reader may also wish to consult \cite[Chapter 7]{wilson} for more on completed group rings of profinite groups over profinite rings in the general setting, or \cite[Ch. II, \S 2.2]{lazard} for more on integral Iwasawa algebras of $p$-valuable groups.
\end{defn}

\begin{rks}\label{rk: kG is profinite}
$ $

\begin{enumerate}[label=(\roman*)]
\item
$kG$ is always profinite. Indeed, if $k$ is finite, then each $k[G/U]$ is a finite ring. If $k$ is a complete discrete valuation ring with uniformiser $\pi$, then $kG$ may be identified with $\invlim (k/\pi^n k)[G/U]$, where the inverse limit ranges over the inverse system of open normal subgroups $U\lhd G$ and $n\geq 1$: see \cite[\S 7.1]{DDMS}. As the rings $(k/\pi^n k)[G/U]$ are all finite, it follows that $kG$ is profinite.
\item It follows from Lemma \ref{lem: a profinite ring is cpt t.d.} and (i) that $kG$ is compact and totally disconnected.
\item When $G$ is $p$-valuable, $kG$ admits a complete filtration by powers of its Jacobson radical, and the associated graded ring is noetherian: see e.g. \cite[\S 6.2]{ardakovInv} or \cite[III, \S 2.3]{lazard}. Hence by \cite[p. 85, Corollary 5]{LVO} every ideal of $kG$ is closed. Combined with (ii), this says that we may apply Lemma \ref{lem: inv lim of quotients} to arbitrary ideals $I$ of $kG$.
\item When $G$ is $p$-valuable, the subring $k[G]$ (the ordinary group ring of $G$ over $k$) is dense in $kG$ \cite[II, 2.2.1--3]{lazard}.
\item If $G$ is pro-$p$, then $kG$ is a local ring \cite[Proposition 7.5.3]{wilson}.
\end{enumerate}
\end{rks}

\subsection{Ideals}

Let $k$ be a finite field of characteristic $p$ or a complete discrete valuation ring whose residue field is a finite field of characteristic $p$. Let $G$ be a $p$-valuable group.

Due to Remark \ref{rk: kG is profinite}(iii), we will say simply ``ideal" instead of ``closed ideal" throughout the remainder of this paper.

\begin{defn}
Let $H$ be a closed normal subgroup of $G$, and $I$ an ideal of $kH$. Then we will say $I$ is \emph{$G$-invariant} if it is invariant under conjugation by $G$. Furthermore, we will say that $I$ is \emph{$G$-prime} if it is $G$-invariant and has the property that, whenever $A$ and $B$ are $G$-invariant ideals of $kH$ with $AB\subseteq I$, we must have either $A\subseteq I$ or $B\subseteq I$.
\end{defn}

\begin{rk}\label{rk: filtered module}
If $I$ is a $G$-invariant ideal of $kH$, and $H$ is a closed normal subgroup of $G$, then $IkG$ is an ideal of $kG$.

If $H$ is also isolated normal, then $kG/IkG \cong (kH/I)[[x_1, \dots, x_e]]$ as a filtered $kH/I$-module: this may be proved along the same lines as the very similar results \cite[Proposition 3.6]{ardakovContemp} and \cite[Lemma 8.5]{ardakovInv}.
\end{rk}

\begin{defn}
If $I$ is an ideal of $kG$, we write $I^\dagger = \ker(G\to kG^\times\to (kG/I)^\times) = (I+1)\cap G$, and we say that $I$ is \emph{faithful} if $I^\dagger = \{1\}$.
\end{defn}

Note that, as $G$ is assumed $p$-valuable, it follows that $I^\dagger$ is an isolated normal closed subgroup of $G$ by \cite[Lemma 5.3(a), (c)]{ardakovInv}.

\begin{lem}
Let $H$ be an isolated normal closed subgroup of $G$, and let $I$ be a faithful $G$-invariant ideal of $kH$. Then $IkG$ is a faithful ideal of $kG$.
\end{lem}

\begin{proof}
The proof of \cite[8.6(b)]{ardakovInv} carries over to this situation \emph{mutatis mutandis}.
\end{proof}

The following is also well known:

\begin{lem}
Let $N$ be a closed normal subgroup of $G$. Then the augmentation ideal $w_N := \ker(kG \to k[[G/N]])$ is generated as a right ideal by $N-1$.
\end{lem}

\begin{proof}
This follows from \cite[Proposition 7.1.2]{wilson}, after taking Remark \ref{rk: kG is profinite}(iii) into account.
\end{proof}

\begin{lem}\label{lem: computation with augmentation ideals}
Let $N$ be an open normal subgroup of $G$, and write $N_2 = [G,N]$. Then, for all $r\in kN$ and $g\in G$, we have $grg^{-1} \in kN$ and $gr-rg\in (N_2-1)kG = w_{N_2}$.
\end{lem}

\begin{proof}
Let $n\in N$ and $g\in G$. Then $gng^{-1}n^{-1} \in N_2$ implies $gng^{-1}\in N\subseteq k[N]$ and $gn - ng = (gng^{-1}n^{-1} - 1)ng \in (N_2-1)k[G]$. It now follows for all $r\in k[N]$ that $grg^{-1}\in k[N]$ and $gr-rg \in (N_2-1)k[G]$ (as $r$ is just a finite $k$-linear combination of elements $n\in N$), and the lemma follows by taking closures inside $kG$ as in Remark \ref{rk: kG is profinite}(iv).
\end{proof}

\section{Inner automorphisms}

\subsection{Uniform subgroups}\label{subsec: uniform groups}

Before we prove Theorem A for general $p$-valuable groups $G$, we first prove the analogue for \emph{uniform} groups $G$, and so we recall below the relevant definitions and facts for the convenience of the reader.

\begin{defn} A (topologically) finitely generated pro-$p$ group is called \emph{powerful} if $[G,G] \leq G^p$ if $p$ is odd, or $[G,G]\leq G^4$ if $p=2$. \cite[Definition 3.1, Lemma 3.4]{DDMS}

A \emph{uniformly powerful} pro-$p$ group, or a \emph{uniform} group for short, is a powerful finitely generated pro-$p$ group which is torsion-free \cite[Theorem 4.5]{DDMS}.
\end{defn}

\begin{props}\label{props: uniform groups}
Let $G$ be uniform. Then the following properties hold:

\begin{enumerate}[label=(\roman*)]
\item For each $i\geq 1$, set $G_i = G^{p^{i-1}}$: this is a characteristic open subgroup of $G$ \cite[Theorem 3.6(ii)]{DDMS}. The filtration $G = G_1 > G_2 > \dots$ is a fundamental system of neighbourhoods of the identity for $G$ \cite[Proposition 1.16(iv)]{DDMS}, and also defines an integer-valued $p$-valuation on $G$ in the sense of \cite[III, 2.1.2]{lazard}. We also have $[G,G_i] \leq G_{i+1}$ \cite[Theorem 3.6(i)]{DDMS}.
\item There exists a binary operation $+$ on $G$ which makes $G$ into a $\mathbb{Z}_p$-module; moreover, any topological generating set for $G$ as a group is also a basis for $G$ as a $\mathbb{Z}_p$-module \cite[Definition 4.12, Theorem 4.17]{DDMS}. In additive notation, the subgroup $G_n$ defined above is equal to the $\mathbb{Z}_p$-submodule $p^{n-1} G$ \cite[Lemma 4.14(iii)]{DDMS}.
\item Let $H$ be an isolated normal closed subgroup of $G$. Then there exists a topological generating set $\{a_1, \dots, a_d\}$ for $G$ such that the subset $\{a_1, \dots, a_e\}$ (for some $e\leq d$) is a topological generating set for $H$: cf. \cite[proof of Lemma 8.5]{ardakovInv} or \cite[remark after Corollary 4.4]{woods-struct-of-G}.
\end{enumerate}
\end{props}

\begin{lem}\label{lem: action of uniform group}
Let $G$ be a uniform group and $H$ an isolated normal closed subgroup. Then $H_n = G_n \cap H$ and $[G,H_i] \subseteq H_{i+1}$.
\end{lem}

\begin{proof}
Choose a topological generating set $\{a_1, \dots, a_d\}$ for $G$ such that $\{a_1, \dots, a_e\}$ is a topological generating set for $H$ as in Property \ref{props: uniform groups}(iii).

Now consider $H$ and $G_n$ as $\mathbb{Z}_p$-modules. By Property \ref{props: uniform groups}(ii), $G_n$ has $\mathbb{Z}_p$-basis $\{p^n a_1, \dots, p^n a_d\}$, and so $G_n\cap H$ has basis $\{p^n a_1, \dots, p^n a_e\}$ by linear algebra. This shows that $G_n\cap H = H_n$.

Reverting to group notation: $H$ normal in $G$ by assumption, and so
$$[G, H_i] = [G, G_i \cap H] \leq [G, G_i] \cap [G, H] \leq G_{i+1} \cap H = H_{i+1}$$
using Property \ref{props: uniform groups}(i) as required.
\end{proof}

\begin{propn}\label{propn: from H_(i+1) to H_(i+2)}
Let $H$ be a uniform group, and fix some $i\geq 1$ and $a\in kH_i^\times$. Define the inner automorphism $\sigma$ of $kH$ by $\sigma(r) = ara^{-1}$. Let $T$ be a transversal to $H_{i+1}$ in $H_i$, so that we may write $a = \sum_{x\in T} a_x x$ for coefficients $a_x\in kH_{i+1}$.
\begin{enumerate}[label=(\roman*)]
\item If $\sigma(h)h^{-1}\in H_{i+1}$ for some $h\in H$, then $a_x = \sigma(h)a_x xh^{-1}x^{-1}$ for all $x\in T$.
\item If $\sigma(h)h^{-1}\in H_{i+1}$ for \emph{all} $h\in H$, then there exist an element $y\in H_i$ and $b\in kH_{i+1}^\times$ with the following properties. Define the inner automorphism $\tau$ of $kH$ by $\tau(r) = brb^{-1}$: then $\sigma(h) = \tau(yhy^{-1})$ and $\tau(h)h^{-1}\in H_{i+2}$ for all $h\in H$.
\end{enumerate}
\end{propn}

\begin{proof}
$ $

\begin{enumerate}[label=(\roman*)]
\item
Suppose we are given such an $h\in H$. Rearranging the equation $\sigma(h) = aha^{-1}$ gives $a = \sigma(h)ah^{-1}$, which we may rewrite as
\begin{equation}\label{eqn: a = sigma(h) a h^(-1)}
a = \sum_{x\in T} a_x x = \sum_{x\in T} \sigma(h)(a_x x)h^{-1} = \sum_{x\in T} (\sigma(h) a_x xh^{-1}x^{-1})x.
\end{equation}
As $kH_i = \bigoplus_{x\in T} (kH_{i+1})x$, we will be done if we can show that $\sigma(h) a_x xh^{-1}x^{-1}\in kH_{i+1}$. But
\begin{itemize}
\item $\sigma(h) a_x \sigma(h)^{-1}\in k{H_{i+1}}$ as $H$ normalises $k{H_{i+1}}$,
\item $\sigma(h) h^{-1} \in H_{i+1} \subseteq k{H_{i+1}}$ by assumption, and
\item $hxh^{-1}x^{-1}\in H_{i+1} \subseteq k{H_{i+1}}$ by Lemma \ref{lem: action of uniform group},
\end{itemize}
and the product of these three elements is $\sigma(h) a_x xh^{-1}x^{-1}$.
\item Since $a$ is invertible, some $a_x$ must also be invertible, as $kH_i$ is local. So fix $y\in T$ such that $a_y$ is invertible. Define $\tau\in\Inn(k{H_{i+1}})$ to be conjugation (on the left) by $a_y$, and define $\eta\in \Inn(H)$ to be conjugation (on the left) by $y$. Then eqn (\ref{eqn: a = sigma(h) a h^(-1)}) for $x=y$ says that $\sigma(h) = a_y y h y^{-1} a_y^{-1} = \tau(yhy^{-1})$ for all $h\in H$.

It remains to show that $\tau(h)h^{-1} \in H_{i+2}$. But $\tau(h)h^{-1} = a_yha_y^{-1}h^{-1}$, and so by Lemma \ref{lem: computation with augmentation ideals}, $\tau(h)h^{-1} - 1 = (a_yh-ha_y)a_y^{-1}h^{-1} \in w_{H_{i+2}}$. This means that $\tau(h)h^{-1} \in w_{H_{i+2}}^\dagger = H_{i+2}$.\qedhere
\end{enumerate}
\end{proof}

\begin{thm}\label{thm: inner automs for uniform groups}
Let $H$ be a uniform group, and let $\sigma$ be an inner automorphism of $kH$ that fixes $H$. Then $\sigma$ is the unique extension of an inner automorphism of $H$.
\end{thm}

\begin{proof}
Write $\sigma_1 = \sigma$. First, note as in the previous proof that $\sigma_1(h)h^{-1} - 1 \in w_{H_2}$ by Lemma \ref{lem: computation with augmentation ideals}, and so in particular $\sigma_1(h)h^{-1} \in w_{H_2}^\dagger = H_2$.

We now define $\sigma_n$ and $y_n$ inductively for all $n\geq 1$. Given an inner automorphism $\sigma_i$ of $kH$ such that $\sigma_i(h)h^{-1}\in H_{i+1}$ for all $h\in H$, Proposition \ref{propn: from H_(i+1) to H_(i+2)}(ii) will give an element $y_{i}\in H_i$ and an inner automorphism $\sigma_{i+1}$ of $kH$ such that $\sigma_i(h) = \sigma_{i+1}(y_{i}hy_{i}^{-1})$ and $\sigma_{i+1}(h)h^{-1}\in H_{i+2}$ for all $h\in H$. Now set $g_0 = 1$ and $g_{i+1} = y_{i+1} g_i$ for all $i\geq 0$, so that $\sigma(h) = \sigma_{i+1}(g_ihg_i^{-1})$ for all $h\in H$ and all $i\geq 0$. Then the sequence $(g_n)_{n\geq 0}$ converges inside $H$, say to some limit $g_\infty$, and $\sigma(h) = g_\infty hg_\infty^{-1}$ for all $h\in H$.
\end{proof}

\subsection{The general result}

Now let $G$ be a $p$-valuable group. The argument below is an extension of a standard argument for finite $p$-groups.

\textit{Proof of Theorem A.}
$G$ contains a characteristic uniform open subgroup $N$ \cite[Corollary 4.3]{DDMS}, and so we may write $kG = kN * F$ (a crossed product in the sense of \cite{passmanICP}), where $F$ is a finite $p$-group. Suppose that $\sigma$ is an inner automorphism of $kG$ which preserves $G$, say there exists $a\in kG^\times$ such that $\sigma(r) = ara^{-1}$ for all $r\in kG$.

This time, a slightly different notation will be convenient. Let $T$ be a transversal to $N$ in $G$, and for each $x\in G$, write $n_x$ and $t_x$ for the unique elements $n_x\in N$ and $t_x\in T$ satisfying $n_x t_x = x$. We will also index the coefficients of $a$ using the cosets, i.e. we will write 
$$a = \sum_{x \in T} a_{xN} \cdot x,$$
for some coefficients $a_{xN} \in kN$. In particular, if $T'$ is any (perhaps different) transversal, we still have
$$a = \sum_{x\in T} a_{xN} \cdot t_x = \sum_{y\in T'} a_{yN} \cdot t_y.$$

Given $g, x\in G$, we will write $g\hash x = \sigma(g) xg^{-1}$; note that $(g, xN) \mapsto (g\hash x)N$ is a well-defined left action of $G$ on the \emph{set} of cosets $G/N$. As $G$ is pro-$p$ and $G/N$ is a finite set, the orbits $G\hash xN$ must all have cardinality a power of $p$.

Fix an arbitrary $g\in G$. On the one hand, $a = \sigma(g)ag^{-1}$, or in other words
$$\sum_{x\in T} a_{xN}\cdot x = \sum_{x\in T} \sigma(g) a_{xN}\cdot x g^{-1} = \sum_{x\in T}  a_{xN}^{\sigma(g)^{-1}} n_{g\hash x}\cdot t_{g\hash x}.$$
On the other hand, by making the change of variables $x \mapsto g\hash x$, we can see that
$$\sum_{x\in T} a_{xN}\cdot x = \sum_{g\hash x \in g\hash T} a_{g\hash xN}\cdot t_{g\hash x}.$$
As before, this implies that
\begin{equation}\label{eqn: g hash x}
a_{xN}^{\sigma(g)^{-1}} n_{g\hash x} = a_{g\hash xN}
\end{equation}
for all $xN\in G/N$ and (as $g$ was arbitrary) for all $g\in G$.

Now consider the augmentation map $\varepsilon: kG \to k$ sending all $g\in G$ to $1$: applying $\varepsilon$ to equation (\ref{eqn: g hash x}) gives $\varepsilon(a_{xN}) = \varepsilon(a_{g\hash xN})$, and so in particular, if the orbit $G\hash yN$ of $yN\in G/N$ has size $p^r$, then
$$\sum_{xN \in G\hash yN} \varepsilon(a_{xN}) = p^r \varepsilon(a_{yN}).$$
But we can calculate $\varepsilon(a)$ as
$$\varepsilon(a) = \varepsilon\left( \sum_{x\in T} a_{xN} x\right) = \sum_{xN\in G/N} \varepsilon(a_{xN}) = \sum_{\substack{\text{distinct}\\ \text{orbits} \\ G\hash yN}}\left( \sum_{xN \in G\hash yN} \varepsilon(a_{xN})\right),$$
and note that this must be a unit as $a$ is invertible. In particular, some orbit $G\hash yN$ must have size 1, and of course we may choose this element $y$ to be an element of $T$, so that $t_y = t_{g\hash y} = y$ for all $g\in G$.

For this particular $y\in G$, equation (\ref{eqn: g hash x}) reads $a_{yN}^{\sigma(g)^{-1}} n_{g\hash y} = a_{yN}$ for all $g\in G$, and this can be rearranged to give that $\sigma(g) = a_{yN} ygy^{-1} a_{yN}^{-1}$. It now follows that the automorphism $g\mapsto \sigma(y^{-1}gy)$ of $kG$ is given by conjugation by an element of $kN^\times$.

Call this automorphism $\beta$. As $\sigma$ preserves $G$, so does $\beta$, and as $N$ is characteristic in $G$, $\beta$ must also preserve $N$. In particular, the restriction $\beta|_{kN}\in \Aut(kN)$ satisfies the hypotheses of Theorem \ref{thm: inner automs for uniform groups}: that is, there exists some $n\in N$ such that $\beta(h) = nhn^{-1}$ for all $h\in N$. But now, for every $g\in G$, there exists some $r$ such that $g^{p^r}\in N$, and so $\beta(g)^{p^r} = (ngn^{-1})^{p^r}$; and as $G$ is $p$-valuable, the desired result now follows from Property \ref{props: p-valuable groups}(v).
\qed

\textit{Proof of Corollary B.} As in the introduction, fix any element $g$ satisfying $G = \overline{\langle H, g\rangle}$: then setting $x = g - 1$ and letting $\sigma$ be conjugation by $g$ gives an equality of filtered left $kH$-modules $kG = kH[[x; \sigma, \sigma-\mathrm{id}]]$. If $\sigma$ is inner, then by Theorem A there exists $h\in H$ such that $grg^{-1} = hrh^{-1}$ for all $r\in kH$. But this means that $h^{-1}g$ centralises $H$, and so $G = \overline{\langle H, h^{-1}g\rangle} = H \times \overline{\langle h^{-1}g\rangle}$.\qed

\section{The action of $\Gamma$}

\subsection{Centralised group elements}

\checked\begin{lem}\label{lem: central group hom}
Let $G$ be a group, $\Gamma$ a subgroup, and $H$ a normal subgroup of $G$ centralising $\Gamma$. Suppose, for some fixed $g\in G$, we have that $x^{-1}gx\in Hg$ for all $x\in\Gamma$; that is, the function $\theta := \theta_g: \Gamma\to G$ defined by $\theta(x) = x^{-1}gxg^{-1}$ has image in $H$. Then $\theta$ is in fact a group homomorphism with image contained in $Z(H)$.
\end{lem}

\begin{rk}
This is mentioned (without proof) in the context of linear groups by Roseblade in \cite{roseblade}; the proof is easy, but we have chosen to supply it.
\end{rk}

\begin{proof}
To see that $\theta$ is a group homomorphism, note that
$$\theta(x)\theta(y) = \boxed{x^{-1}g x g^{-1}}\,\boxed{y^{-1}}\,g y g^{-1},$$
and that the two boxed factors commute, as the first is an element of $H$ and the second is an element of $\Gamma$. This expression then simplifies to $\theta(x)\theta(y) = y^{-1}x^{-1}g xy g^{-1} = \theta(xy)$ as required. To see that the image is central in $H$, we compute, for arbitrary $h\in H$,
$$h^{-1} \theta(x) h = \boxed{h^{-1}}\, \boxed{x^{-1}}\, g x g^{-1} h,$$
noting again that the boxed expressions commute for the same reason as above. Now, as $H$ is normal in $G$, there exists some $h'\in H$ such that $h^{-1} g = gh'^{-1}$ and hence that $g^{-1} h = h'g^{-1}$. Making these substitutions, we can see that
$$h^{-1} \theta(x) h = x^{-1} gh'^{-1} x h'g.$$
Finally, $h'^{-1}xh' = x$, and so this equation reduces to $h^{-1} \theta(x) h = \theta(x)$.
\end{proof}

\textbf{For the remainder of this subsection,} let $G$ be a $p$-valuable group, $H$ a closed isolated normal subgroup (so that $G/H$ is $p$-valuable by Property \ref{props: p-valuable groups}(iv)), and $\Gamma$ a closed normal subgroup of $G$ acting on $G$ by conjugation and centralising $H$. If $I$ is a $G$-invariant ideal of $kH$, we will always view $kG/IkG$ as a filtered $kH/I$-module as in Remark \ref{rk: filtered module}.

\checked\begin{propn}\label{propn: finite to singleton}
Let $I$ be a $G$-invariant ideal of $kH$. Suppose that $y = a_1 g_1 + \dots + a_n g_n\in kG/IkG$ is centralised by $\Gamma$, where each $a_i\in kH/I$, and the $g_i$ are elements of $G$ which are pairwise distinct modulo $H$. Then $a_i g_i$ is centralised by $\Gamma$ for each $1\leq i\leq n$.
\end{propn}

\begin{proof}
$\Gamma$ permutes the submodules $(kH/I)g_i$, which means that it must permute the summands $a_ig_i$ and also the cosets $Hg_i$ in the same way. We will prove that, if $Hg\in G/H$ has \emph{finite} $\Gamma$-orbit, then it has \emph{trivial} $\Gamma$-orbit: this will suffice.

Let $S$ be the stabiliser in $\Gamma$ of $Hg$. Then $S$ has finite index in $\Gamma$, so its normal core $S^\circ := \bigcap_{\gamma\in\Gamma} S^\gamma$ has finite index in $\Gamma$ too, say index $p^f$, as $\Gamma$ is orbitally sound by Properties \ref{props: p-valuable groups}(iii), (iv). Hence, for all $x\in\Gamma$, we have $x^{p^f}\in S^\circ \subseteq S$, and so
$$(Hg^{-1}xg)^{p^f} = Hg^{-1}x^{p^f}g = Hx^{p^f},$$
an equality inside $G/H$. But now applying Property \ref{props: p-valuable groups}(v) to this equation shows that $x$ fixes $Hg$, i.e. $S = \Gamma$.
\end{proof}

\checked\begin{propn}\label{propn: ring elt to group elt}
Let $I$ be a faithful $G$-prime ideal of $kH$. Let $g\in G$ be such that $x^{-1}gx\in Hg$ for all $x\in\Gamma$. Suppose that, for some nonzero $a\in kH/I$, the element $ag\in kG/IkG$ is centralised by $\Gamma$. Then $g$ is also centralised by $\Gamma$.
\end{propn}

\begin{proof}
Since $I$ is faithful, we may view $G$ as a subgroup of $(kG/IkG)^\times$. Define $\theta: \Gamma\to G$ by $\theta(x) = x^{-1}gxg^{-1}$, so that by Lemma \ref{lem: central group hom} we have $\theta(x) \in Z(H)$ for all $x\in \Gamma$. We will prove that in fact $\theta(x) = 1$ for all $x\in \Gamma$.

Write $I$ as the intersection of a finite $G$-orbit of prime ideals $J_1, \dots, J_s$ of $kH$, which is possible due to \cite[Theorem 3.18]{letzter-noeth-skew}. As $a$ is assumed nonzero inside $kH/I$, we can take $a = \tilde{a} + I$ for some $\tilde{a}\in kH\setminus I$; and $\tilde{a}\not\in I$ implies that $\tilde{a}\not\in J_i$ for at least one $1\leq i\leq s$. Choose such a $J_i$, and denote it simply by $J$.

Take $x\in \Gamma$ with $\theta(x) = z\in Z(H)$. This means that
$$ag = x^{-1}(ag) x = zag,$$
so that $(z-1)\tilde{a}\in I\subseteq J$. Now, as $z$ is central in $H$, this implies that $(z-1)\cdot kH\cdot \tilde{a}\subseteq J$; and since $J$ is prime, we must now have $z-1\in J$, i.e. $z\in J^\dagger$.

As the $G$-orbit of $J$ is finite, the closed subgroup $J^\dagger$ is orbital in $G$, and we may calculate its normal core in $G$ as $(J^\dagger)^\circ = \bigcap_{g\in G} (J^\dagger)^g = J_1^\dagger\cap \dots\cap J_s^\dagger = I^\dagger$, which is $\{1\}$ by assumption. But now Property \ref{props: p-valuable groups}(ii) tells us that the index $[J^\dagger:(J^\dagger)^\circ]$ is finite, and so $J^\dagger$ is a finite subgroup of $G$. Finally, $G$ is torsion-free by Property \ref{props: p-valuable groups}(i), so $J^\dagger = \{1\}$. In particular, this tells us that $z = 1$ as required.
\end{proof}

At this stage, we note the following for later use. Let $H$ be the centraliser in $G$ of $\Gamma$. Let $T$ be an arbitrary transversal to $H$ in $G$, and write $R = \bigoplus_{t\in T} (kH/I)t$. (Note that $R$ is a \emph{subring} of $kG/IkG$, as $H$ is normal in $G$.)

\begin{propn}\label{propn: fixed rings of partial completions}
If $I$ is a faithful $G$-prime ideal of $kH$, then $\C_\Gamma(R) = kH/I$.
\end{propn}

\begin{proof}
The inclusion ``$\supseteq$" is clear.

For the converse, suppose that the element $y = a_1 t_1 + \dots + a_n t_n\in R$ is centralised by $\Gamma$, where each $a_i\in kH/I$ is nonzero, and the $t_i$ are distinct elements of $T$. Then, as $G/H$ is $p$-valuable by Lemma \ref{lem: G/H is p-val}, we may apply Proposition \ref{propn: finite to singleton} to see that the elements $a_it_i\in R$ are centralised by $\Gamma$ for each $1\leq i\leq n$.

Now let $x\in \Gamma$ and $1\leq i\leq n$ be arbitrary, and write $a = a_i$ and $t = t_i$. Set $t^x = s\in G$. Then $at = (at)^{x} \in ((kH/I) t)^{x} = (kH/I) s$ implies (by intersecting both sides with the naturally embedded copy of $G$) that $Hs = Ht$, and so in particular $x^{-1}tx = s \in Hs = Ht$. As $x$ was arbitrary, Proposition \ref{propn: ring elt to group elt} shows that $t$ is centralised by $\Gamma$, i.e. $t\in H$. And as $i$ was also arbitrary, and the $t_i$ were assumed distinct mod $H$, we must have $n=1$ and $y = a_1t_1\in kH/I$.
\end{proof}

\subsection{Choosing a compatible family of transversals}\label{subsec: transversals}

Let $G$ be a $p$-valuable group, and $H$ a closed, isolated, normal subgroup of $G$.

In what follows, it will be convenient to fix a descending sequence $G = G_1 \supseteq G_2 \supseteq G_3 \supseteq \dots$ of open normal subgroups of $G$ such that $\bigcap_{n\geq 1} G_n = \{1\}$. This can be done using e.g. \cite[Proposition 1.16(iii)]{DDMS}. This sequence (viewed as an inverse system over the directed set of positive integers) is cofinal with the inverse system of \emph{all} open normal subgroups of $G$, and so $G \cong \invlim_i (G/G_i)$ by Properties \ref{props: profinite objects}(ii), (iii). For each $i\geq 1$, and each closed subgroup $K$ of $G$, write $K_i = G_i\cap K$. 

In this section, we will choose a family of ``compatible" transversals to $H_i$ in $G_i$ for all $i \geq 1$, and use them to construct a transversal to $H$ in $G$ compatible with the quotient maps. In particular, we will write the following:

\centerline{
\xymatrix{
G\ar@{=}[r]\ar[d]_{\rho}&\invlim_i (G/G_i)\ar[r]^-{\varphi_n}&G/G_n\ar[r]^{\pi_n}\ar[d]_{\rho_n}& G/G_{n-1}\ar[d]^{\rho_{n-1}}\\
G/H\ar@{=}[r]&\invlim_i (G/HG_i)\ar[r]_-{\overline{\varphi}_n}&G/HG_n\ar[r]_{\overline{\pi}_n}& G/HG_{n-1}
}
}

Here, $\rho, \rho_n, \pi_n, \overline{\pi}_n$ are the natural quotient maps for each $n$, the inverse limits run over $\pi_n$ and $\overline{\pi}_n$, and $\varphi_n, \overline{\varphi_n}$ are the natural projection maps. This diagram commutes.

A transversal to $H$ in $G$ is a subset $T\subseteq G$ such that, for each $g\in G$, there exists a unique $t\in T$ satisfying $Hg = Ht$: in other words, we want a map of sets $\tau: G/H\to G$ splitting the surjective map $\rho$, so that $T = \mathrm{im}(\tau)$.

\checked\begin{propn}\label{propn: transversals}
There exists a sequence of transversals $T_n$ to $HG_n/G_n$ in $G/G_n$ such that $\pi_n(T_n) = T_{n-1}$ for all $n$.
\end{propn}

\begin{proof}
We proceed by induction. The case $n = 1$ is trivial: $G_1 = G$, so $HG_1/G_1 = G/G_1 = \{G_1\}$ is the trivial group, which is its own transversal, so we set $T_1 = G/G_1$.

Suppose now that $n > 1$, and that $T_{n-1}$ is a transversal to $HG_{n-1}/G_{n-1}$ in $G/G_{n-1}$: that is, we have a splitting $\tau_{n-1}$ of the surjective map $\rho_{n-1}$. Choose an arbitrary splitting $\sigma$ of the surjective map $\pi_n$, and define the map $\tau_n = \sigma \tau_{n-1} \overline{\pi}_n: G/HG_n\to G/G_n$. A diagram chase shows that
\begin{align*}
\tau_n \rho_n \tau_n &= (\sigma \tau_{n-1} \overline{\pi}_n) \rho_n (\sigma \tau_{n-1} \overline{\pi}_n) & \text{by definition}\\
&= \sigma \tau_{n-1} \rho_{n-1} \pi_n \sigma \tau_{n-1} \overline{\pi}_n & \text{by commutativity of the square}\\
&= \sigma \tau_{n-1} \rho_{n-1} \tau_{n-1} \overline{\pi}_n & \text{as } \sigma \text{ splits } \pi_n\\
&= \sigma \tau_{n-1} \overline{\pi}_n & \text{as } \tau_{n-1} \text{ splits } \rho_{n-1}\\
&= \tau_n &\text{by definition,}
\end{align*}
and as $\tau_n$ is injective, it is left-cancellable, so that $\rho_n \tau_n = \mathrm{id}_{G/HG_n}$. This tells us that $\tau_n$ does in fact split $\rho_n$, and we may set $T_n = \mathrm{im}(\tau_n)$. Now $\pi_n(T_n) = \mathrm{im}(\pi_n \tau_n) = \mathrm{im}(\pi_n \sigma \tau_{n-1} \overline{\pi}_n)$, and as $\pi_n\sigma = \mathrm{id}_{G/G_{n-1}}$ and $\overline{\pi}_n$ is surjective, this is just equal to $\mathrm{im}(\tau_{n-1}) = T_{n-1}$.
\end{proof}

\checked\begin{cor}\label{cor: inverse limit of transversals}
The inverse limit $T := \invlim_i(T_i)$ (where the $T_i$ are viewed as discrete topological spaces) is a transversal to $H$ in $G$, and $\varphi_n(T) = T_n$ for all $n$.\qed
\end{cor}

We will make crucial use of this transversal later.

\section{The lifting procedure}

\subsection{Twisted permutation modules}

The proof of Theorem C essentially follows \cite[Propositions 2.1 and 2.2]{ardakovDocumenta}, but as we are no longer dealing with straightforward permutation modules, it requires some extra setup.

\begin{defn}
Let $A$ be a ring, $M$ a free left $A$-module with basis $X$, and $\Gamma$ a group acting $A$-linearly on $M$. If, for every $\gamma\in\Gamma$ and $x\in X$, there exist unique $x'\in X$ and $a\in A^\times$ such that $x' = ax^\gamma$, we call $M$ a \emph{twisted permutation module} for $\Gamma$, and denote it $A^t[X]$. We will say $A^t[X]$ has \emph{finite orbits} if the $\Gamma$-orbit of every element of $A^t[X]$ is finite.
\end{defn}

\begin{notn}
Throughout this section, we will have a group $\Gamma$ acting on various sets $S$.
\begin{enumerate}[label=(\roman*),noitemsep]
\item The action of an element $\gamma\in\Gamma$ will always be written as a right superscript, i.e. we will write the action map $S\times \Gamma \to S$ as $(s, \gamma) \mapsto s^\gamma$.
\item The subgroup of $\Gamma$ stabilising the element $s\in S$ will be written $\Gamma_s$. The orbit of $s\in S$ will be written $s^\Gamma$.
\item The subset of $S$ which is centralised by $\Gamma$ will be denoted $\C_\Gamma(S)$: this is non-standard, but we hope that it will cause no confusion in practice.
\item If $\Gamma$ acts on an additive group $A$, and $|a^\Gamma|$ is finite for some $a\in A$, then we will write $\widehat{a}$ for the sum of its $\Gamma$-orbit, i.e. $\sum_{b\in a^\Gamma} b$. Equivalently, if $T$ is any right transversal to $\Gamma_a$ in $\Gamma$, then $\widehat{a} = \sum_{\gamma\in T} a^\gamma$.
\end{enumerate}
\end{notn}

\begin{lem}\label{lem: fixed module is sum}
Let $A^t[X]$ be a twisted permutation module for $\Gamma$ with finite orbits. Then $\C_\Gamma(A^t[X]) = \sum_{x\in X} A\widehat{x}$.
\end{lem}

\begin{proof}
The inclusion ``$\supseteq$" is obvious. For the reverse inclusion: let $r = a_1x_1 + \dots + a_nx_n\in \C_\Gamma(A^t[X])$, where $a_i\in A$ and the $x_i$ are pairwise distinct elements of $X$. Then $\Gamma$ must permute the set $\{a_1x_1, \dots, a_nx_n\}$, and so it falls into disjoint $\Gamma$-orbits $o_1, \dots, o_s$, so that $r = \widehat{o}_1 + \dots + \widehat{o}_s$. It suffices to show that each $\widehat{o}_i$ can be written as $a\widehat{x}$ for some $a\in A$ and $x\in X$. To this end, suppose (renumbering without loss of generality) that $o_i = \{a_1x_1, \dots, a_mx_m\}$: but then $a_1x_1 + \dots + a_mx_m = \widehat{a_1x_1} = a_1\widehat{x}_1$.
\end{proof}

Given any $x\in X$, we will denote by $M(x)$ the submodule $\sum_{\gamma\in\Gamma} Ax^\gamma\subseteq A^t[X]$.

\begin{lem}\label{lem: hat submodules are equal or disjoint}
Let $A^t[X]$ be a twisted permutation module for $\Gamma$ with finite orbits, and let $x_1$ and $x_2\in X$. If there exist $a\in A^\times$ and $h\in\Gamma$ such that $x_1 = ax_2^h$, then $\widehat{x}_1 = a\widehat{x}_2$ and $A\widehat{x}_1 = A\widehat{x}_2$. If not, then $M(x_1)\cap M(x_2) = \{0\}$.
\end{lem}

\begin{proof}
Let $\mathcal{S} = \{Ax\}_{x\in X}$ be the set of $A$-submodules of $A^t[X]$ each generated by a single element of $X$: note that the definition of twisted permutation module ensures that $\Gamma$ acts on $\mathcal{S}$.

If $Ax_1$ and $Ax_2$ are in the same orbit, then $Ax_1 = Ax_2^h$ for some $h\in\Gamma$, and so $x_1 = ax_2^h$ for some $a\in A$; similarly $Ax_2 = Ax_1^{h^{-1}}$ and so $x_2 = bx_1^{h^{-1}}$ for some $b\in A$, implying that $x_1 = abx_1$ and so $a$ is invertible. In this case, $\Gamma_{x_2} = h\Gamma_{x_1} h^{-1}$, and so if $T$ is a right transversal to $\Gamma_{x_1}$ in $\Gamma$, then $hT$ is a right transversal to $\Gamma_{x_2}$ in $\Gamma$. Hence
$$\widehat{x}_1 = \sum_{\beta \in T} x_1^\beta = \sum_{\beta \in T} ax_2^{h\beta} = \sum_{h\beta \in hT} ax_2^{h\beta} = a\widehat{x}_2,$$
and so $A\widehat{x}_1 = A\widehat{x}_2$. Otherwise, $Ax_1$ and $Ax_2$ are in disjoint orbits: as the $x\in X$ are $A$-linearly independent by definition, this ensures that $M(x_1)$ and $M(x_2)$ intersect only in the zero submodule.
\end{proof}

\begin{rk}\label{rk: projections}
We draw an immediate conclusion from Lemmas \ref{lem: fixed module is sum} and \ref{lem: hat submodules are equal or disjoint}. Elements of $\C_\Gamma(A^t[X])$ can be written as $a_1 \widehat{x}_1 + \dots + a_r \widehat{x}_r$, for coefficients $a_i\in A$ and elements $x_i\in X$, and the $x_i$ can be chosen so that the $M(x_i)$ have pairwise zero intersection. Moreover, $\C_\Gamma(A^t[X])$ is the direct sum of the $M(x_i)$ (with duplicates removed), so for each $x\in X$ there is a well-defined projection map $\pr_{M(x)} : \C_\Gamma(A^t[X]) \to M(x)$.

Note that, by definition, $A\widehat{x} \subseteq M(x)$ for all $x\in X$. In particular, in this language, Lemma \ref{lem: hat submodules are equal or disjoint} says the following: given $x,z\in X$, $\pr_{M(x)}(\widehat{z}) \neq 0$ if and only if there exist $a\in A^\times$ and $h\in \Gamma$ such that $z = ax^h$, and in this case $\pr_{M(x)}(\widehat{z}) = \widehat{z}$.
\end{rk}

\begin{defn}
Let $c\in \C_\Gamma(A^t[X])$ and $x\in X$. We will write $\widehat{x} \in \mathrm{supp}(c)$ if $\pr_{M(x)}(c) \neq 0$.
\end{defn}

\begin{lem}\label{lem: hat counting}
Let $A^t[X]$ and $B^t[Y]$ be twisted permutation modules for $\Gamma$ with finite orbits, and suppose that we have a map of $\Gamma$-modules $f: B^t[Y] \to A^t[X]$. Fix $x\in X$ and $y\in Y$, and suppose that $f(y) = x$. Then $f(\widehat{y}) = n\widehat{x}$, where $n = |y^\Gamma|/|x^\Gamma|$.
\end{lem}

\begin{proof}
As $f$ is a map of $\Gamma$-sets, we have $\Gamma_y\subseteq \Gamma_x$. Then, writing $S$ for a right transversal of $\Gamma_x$ in $\Gamma$ and $T$ for a right transversal of $\Gamma_y$ in $\Gamma$, we have 
$$f(\widehat{y}) = \sum_{\gamma\in T} f(y)^\gamma = \frac{|T|}{|S|}\sum_{\gamma\in S} x^\gamma = \frac{|T|}{|S|}\widehat{x},$$
and $|T| = |y^\Gamma|$ and $|S| = |x^\Gamma|$ by the orbit-stabiliser theorem.
\end{proof}

Finally, we come to the main proposition of the section:

\begin{propn}\label{propn: lifting}
Let $A^t[X]$ and $B^t[Y]$ be twisted permutation modules for $\Gamma$ with finite orbits, and suppose that we have a map of $\Gamma$-modules $f: B^t[Y] \to A^t[X]$. Assume also that $f(Y) \subseteq X$. Suppose that we are given $c\in \C_\Gamma(B^t[Y])$ and $x\in X$ such that $\widehat{x} \in \mathrm{supp}(f(c))$.

If $A$ has characteristic $p>0$ and $\Gamma$ is a pro-$p$ group, then there exists $y\in Y$ such that $f(y) = x$, $|y^\Gamma| = |x^\Gamma|$ and $\widehat{y} \in \mathrm{supp}(c)$.
\end{propn}

\begin{proof}
Let
$$c = b_1 \widehat{y}_1 + \dots + b_s \widehat{y}_s,\qquad f(c) = a_1\widehat{x}_1 + \dots + a_r\widehat{x}_r,$$
where $a_i\in A$, $b_i\in B$, $x_i\in X$ and $y_i\in Y$, such that the $M(x_i)\subseteq A^t[X]$ have pairwise zero intersection and the $M(y_i) \subseteq B^t[Y]$ have pairwise zero intersection. Suppose without loss of generality that $x = x_i$, and write $a = a_i$ so that $\pr_{M(x)}(f(c)) = a\widehat{x}\neq 0$. But
$$\pr_{M(x)}(f(b_1 \widehat{y}_1 + \dots + b_s \widehat{y}_s)) = f(b_1) n_1 \cdot \pr_{M(x)}(\widehat{z}_1) + \dots + f(b_s) n_s \cdot \pr_{M(x)}(\widehat{z}_s),$$
where $f(y_j) = z_j$ and $n_j = |y_j^\Gamma|/|z_j^\Gamma|$, by Lemma \ref{lem: hat counting}. Now there must be some $1\leq j\leq s$ such that $f(b_j)n_j \cdot \pr_{M(x)}(\widehat{z}_j) \neq 0$. For this $j$, we must have $b_j\neq 0$ and $n_j\neq 0$, and by Lemma \ref{lem: hat submodules are equal or disjoint} and Remark \ref{rk: projections}, $\pr_{M(x)}(\widehat{z}_j) \neq 0$ implies that there exist $h\in\Gamma$ and $u\in A^\times$ such that $f(y_j) = z_j = ux^h$, so that $|z_j^\Gamma| = |x^\Gamma|$.

Set $y'$ to be equal to this $y_j$, and let $y\in Y$ be such that $y = v(y')^{h^{-1}}$ for some $v\in B^\times$. It is easy to see that $|y^\Gamma| = |y'^\Gamma|$ and $f(y) = x$, and by definition $\widehat{y}\in \mathrm{supp}(c)$ if and only if $\widehat{y'}\in\mathrm{supp}(c)$.

It remains to show that $|z_j^\Gamma| = |y_j^\Gamma|$, i.e. $n_j = 1$. However, as $\Gamma$ is a pro-$p$ group, $n_j$ must be a power of $p$, and as $A$ has characteristic $p$, all powers of $p$ are either $0$ or $1$ in $A$. As we know $n_j\neq 0$, we must have $n_j = 1$. So we may take $y = y_j$.
\end{proof}

\subsection{Lifting orbits}

Let $G$ be a $p$-valuable group, $\Gamma$ an arbitrary closed normal subgroup which we consider acting on $G$ (on the right) by conjugation, and $H$ the centraliser in $G$ of $\Gamma$. Recall that this implies that $G/H$ is $p$-valuable by Lemma \ref{lem: centralisers are normal}. Fix also a faithful $G$-prime ideal $I\lhd kH$.

We fix further notation below for the remainder of this section.

\begin{notn}
$ $

\begin{enumerate}
\item Fix a descending sequence $G = G_1 \supseteq G_2 \supseteq G_3 \supseteq \dots$ of open normal subgroups of $G$ such that $\bigcap_{n\geq 1} G_n = \{1\}$. For each $i\geq 1$, write $H_i = G_i\cap H$.
\item Write $\varphi_n: G\to G_n$ and $\pi_n: G/G_n \to G/G_{n-1}$ for the natural projection maps.
\item Let $(T_n)_{n\geq 1}$ be a sequence of transversals to $HG_n/G_n$ in $G/G_n$ such that $\pi_n(T_n) = T_{n-1}$ for all $n\geq 2$ as in Proposition \ref{propn: transversals}, and let $T$ be the corresponding inverse limit transversal to $H$ in $G$ from Corollary \ref{cor: inverse limit of transversals}.
\item For each $i\geq 1$, denote $A_i = kH/(w_{H_i} + I)$ and $R_i = kG/(w_{G_i} + IkG)$.
\end{enumerate}
\end{notn}

\begin{rk}
It follows from Lemma \ref{lem: inv lim of quotients} and Remark \ref{rk: kG is profinite}(iii) that the maps $R_{n+1}\to R_n$ form an inverse system with inverse limit $kG/IkG$.
\end{rk}

As before, we view $kG/IkG$ as a filtered left $kH/I$-module.

\checked\begin{defn}\label{defn: R and S}
We define the following left $kH/I$-submodule of $kG/IkG$:
$$R = \bigoplus_{t\in T}(kH/I)t.$$
Note that, as $H$ is \emph{normal} in $G$, in fact $R$ is a \emph{subring} of $kG/IkG$.
\end{defn}

\checked\begin{lem}\label{lem: spanning set of finite rings}
$R_n = A_n^t[T_n]$ is a twisted permutation module for $\Gamma$.
\end{lem}

\begin{proof}
Note first that $H/H_n \cong HG_n/G_n \lhd G/G_n$, so that $k[G/G_n]$ can be written as a crossed product $k[H/H_n]*F$, where $F$ is the finite group $G/HG_n$. Now, writing $\theta: kG\to k[G/G_n]$, we have $\theta(IkG) = \theta(I)k[G/G_n] = \theta(I)\cdot(k[H/H_n]*F)$. Now, as $\theta(I)$ is assumed $F$-stable, this is equal to $\theta(I)*F$ by \cite[Lemma 1.4(ii) and its proof]{passmanICP}. Hence, again by \cite[Lemma 1.4(ii)]{passmanICP}, we have
$$R_n = k[G/G_n]/\theta(IkG) \cong (k[H/H_n]*F){\big /}(\theta(I)*F) \cong A_n*F.$$
Taking $T_n$ as a basis for the crossed product inside $R_n$ yields the desired result.
\end{proof}

Now let $f_i: R_i\to R_{i-1}$ be the natural projection map for each $i\geq 2$. We have the data of an inverse system of $\Gamma$-modules:

\centerline{
\xymatrix{
\dots\ar[r]& R_{n+1}\ar[r]^-{f_{n+1}}& R_n\ar[r]^-{f_{n}}& R_{n-1}\ar[r]& \dots\ar[r]& R_1.
}
}

This gives a cone over the diagram of $\Gamma$-fixed rings:

\centerline{
\xymatrix{
&&\C_\Gamma(kG/IkG)\ar[dl]_-{\varphi_{n+1}}\ar[d]^{\varphi_{n}}\ar[dr]^-{\varphi_{n-1}}\\
\dots\ar[r]& \C_\Gamma(R_{n+1})\ar[r]^-{f_{n+1}}& \C_\Gamma(R_n)\ar[r]^-{f_{n}}& \C_\Gamma(R_{n-1})\ar[r]& \dots
}
}

It is now easy to check that we are in a position to apply Proposition \ref{propn: lifting}.

\begin{thm}
$\C_\Gamma(kG/IkG) = \overline{\C_\Gamma(R)}$.
\end{thm}

\begin{proof}
\textbf{Case 1:} $\mathrm{char}(k) = p > 0$.

Take some nonzero element $\alpha \in \C_\Gamma(kG/IkG)$. Set $\alpha^{(i)} = \alpha$ and proceed inductively for all $i\geq 1$. If $\alpha^{(i)} = 0$, set $\xi^{(i)} = 0$. Otherwise, let $r$ be the minimal positive integer such that $\varphi_r(\alpha^{(i)}) \in \C_\Gamma(R_r)$ is nonzero.

Now, using Proposition \ref{propn: lifting}, we can lift each $x_r\in T_r$ such that $\widehat{x}_r\in \mathrm{supp}(\varphi_r(\alpha^{(i)}))$ to an $x_{r+1}\in T_{r+1}$ whose $\Gamma$-orbit has the same size and such that $\widehat{x}_{r+1}\in \mathrm{supp}(\varphi_{r+1}(\alpha^{(i)}))$. Repeating this procedure gives a sequence of elements $x_r, x_{r+1}, x_{r+2}, \dots$ such that $f_{r+j}(x_{r+j}) = x_{r+j-1}$ for all $j\geq 1$, and so this sequence determines a unique element $x\in T$ whose $\Gamma$-orbit has the same size and such that $\varphi_r(\widehat{x}) = \widehat{x}_r$. Lifting every $\widehat{x}_r\in \mathrm{supp}(\varphi_r(\alpha^{(i)}))$ in this way, we see that there exists $\xi^{(i)}\in \C_\Gamma(R)$ such that $\varphi_r(\xi^{(i)}) = \varphi_r(\alpha^{(i)})$. Now set $\alpha^{(i+1)} = \alpha^{(i)} - \xi^{(i)}$ so that $\varphi_r(\alpha^{(i+1)}) = 0$, and repeat.

Now the sum $\alpha = \xi^{(1)} + \xi^{(2)} + \xi^{(3)} + \dots $ converges in $kG/IkG$, and each term is in $\C_\Gamma(R)$, so $\alpha\in \overline{\C_\Gamma(R)}$.

\textbf{Case 2:} $\mathrm{char}(k) = 0$.

Let $\pi$ be a uniformiser for $k$ (which must now be a complete discrete valuation ring). Take some nonzero element $\alpha\in \C_\Gamma(kG/IkG)$. Set $\alpha^{(0)} = \alpha$ and proceed inductively for all $i\geq 0$ as follows. By reducing $\alpha^{(i)}\in \C_\Gamma(kG/IkG)$ modulo $\pi$, we are back in Case 1, so the argument of Case 1 gives us an element $\beta^{(i)}\in \overline{\C_\Gamma(R)}$ such that $\alpha^{(i)} + \pi \C_\Gamma(kG/IkG) = \beta^{(i)} + \pi \C_\Gamma(kG/IkG)$: that is, there exists some $\alpha^{(i+1)}\in \C_\Gamma(kG/IkG)$ such that $\alpha^{(i)} = \beta^{(i)} + \pi \alpha^{(i+1)}$.

Now the sum $\alpha = \beta^{(0)} + \pi \beta^{(1)} + \pi^2\beta^{(2)} + \dots$ converges in $kG/IkG$, and each term is in $\overline{\C_\Gamma(R)}$, so $\alpha\in\overline{\C_\Gamma(R)}$.
\end{proof}

We can now join this with Proposition \ref{propn: fixed rings of partial completions}.

\begin{thm}\label{thm: fixed ring over F_p}
If $I$ is faithful and $G$-prime, then $\C_\Gamma(kG/IkG) = kH/I$.
\end{thm}

\begin{proof}
Let $R$ be as above, so that $\C_\Gamma(kG/IkG)$ is the closure in $kG/IkG$ of $\C_\Gamma(R)$; and by Proposition \ref{propn: fixed rings of partial completions}, $\C_\Gamma(R) = kH/I$, which is already closed in $kG/IkG$.
\end{proof}

This proves Theorem C.

\section{Applications and open questions}

\subsection{Ideal separation conditions}\label{subsec: ideal separation}

$G$ will continue to denote a $p$-valuable group throughout.

In the usual way, we will say that:
\begin{itemize}
\item $kG$ is \emph{polycentral} if, for all two-sided ideals $A\lneq B$ of $kG$, the ideal $B/A$ contains a nonzero central element of $kG/A$.
\item $kG$ is \emph{centrally separated} if, for all two-sided \emph{prime} ideals $A\lneq B$ of $kG$, the ideal $B/A$ contains a nonzero central element of $kG/A$.
\end{itemize}
(For the reader's convenience, we note that polycentrality is equivalent to \emph{hypercentrality} in this context since $kG$ is right noetherian, and that both terms appear commonly in the literature.)

\textit{Proof of Corollary D.} Let $P$ be a faithful prime ideal of $kG$, and set $I = P\cap kZ$ and $Q = IkG$.

It will suffice to prove that $Q$ is a prime ideal of $kG$. Indeed, by Theorem C (with $\Gamma = G$), we can calculate that $Z(kG/Q) = (Q+kZ)/Q$, and hence $$(P/Q)\cap Z(kG/Q) = (P/Q) \cap (Q+kZ)/Q = (Q + (P\cap kZ))/Q = Q/Q,$$ which is the zero subring of $kG/Q$. (The second equality follows from the modular law.) But $Q \leq P$: so, if $Q$ is prime, as we have assumed that $kG$ is centrally separated, we must have $P = Q$.

So we now prove that $Q$ is a prime ideal of $kG$. Firstly, $I$ is a prime ideal of $kZ$, from which it follows by commutative algebra (see e.g. \cite[Chapter II, Exercise 4.11]{hartshorne}) that the ring $kZ/I$ admits a discrete valuation $v$ with respect to which it is complete. Without loss of generality, let $v$ have image in $\mathbb{N}\cup\{\infty\}$.

Now we may realise $kG/Q$ as an iterated skew power series ring over $kZ/I$ using \cite[3.14(iv)]{letzter-noeth-skew}, and as $G/Z$ is supersoluble by assumption, this iterated extension is \emph{triangular} in the sense of \cite[cf. Example 2.16]{woods-SPS-dim}. We may extend $v$ to a valuation $\tilde{v}$ on $kG/Q$ using the same argument as \cite[Theorem E]{woods-SPS-dim}. This shows that $Q$ is a prime ideal of $kG$.\qed

\subsection{Some parallels}\label{subsec: parallels}

Let $H$ be a \emph{discrete} group and $\mathbb{Z}[H]$ its integral group ring. It is known \cite[Corollary]{roseblade-integral-group-rings} that $\mathbb{Z}[H]$ is polycentral if and only if $H$ is finitely generated and nilpotent, and the paper \cite{roseblade-smith-hypercentral} contains similar results over arbitrary coefficient rings.

A similar fact holds for finite-dimensional Lie algebras $\mathfrak{g}$ over a field of characteristic $0$: in this case, $U(\mathfrak{g})$ is polycentral if and only if $\mathfrak{g}$ is nilpotent \cite[Theorem A.3.3]{jategaonkar}.

The above facts about group rings of polycyclic groups were crucially exploited in \cite[Corollary 11]{roseblade} to understand their full prime ideal structure. We are currently unable to prove an analogue of these facts for Iwasawa algebras, but we can immediately mimic a small part of the argument in our case:

\begin{propn}\label{propn: polycentral soluble implies nilpotent}
Suppose that the $p$-valuable group $G$ contains a nontrivial abelian normal subgroup $A$, and that $\mathbb{Z}_p G$ is polycentral. Then $G$ is nilpotent.
\end{propn}

\begin{proof}
We will show that such a $p$-valuable group must have a nontrivial centre $Z$. Then $\mathbb{Z}_p[[G/Z]]$ will remain polycentral, as it is a quotient of $\mathbb{Z}_p G$; and as $Z$ is known to be isolated in $G$ \cite[Lemma 8.4(a)]{ardakovInv}, it will follow that $G/Z$ is still $p$-valuable \cite[IV, 3.4.2]{lazard} of lower rank, so the statement follows by induction on $\mathrm{rank}(G)$.

Let $\mathfrak{a} = \ker(\mathbb{Z}_p G\to \mathbb{Z}_p[[G/A]])$ and $\mathfrak{g} = \ker(\mathbb{Z}_p G\to \mathbb{Z}_p)$. Consider the function $\theta: A\to \mathfrak{a}/\mathfrak{ga}$ given by $\theta(a) = (1-a) + \mathfrak{ga}$: then a very similar proof to that in \cite[\S 2.2, Lemma 1]{GruRos72} shows that $\theta$ is a $\mathbb{Z}_p$-module isomorphism. (Recall that, if $A$ is written multiplicatively, it is a $\mathbb{Z}_p$-module under \emph{exponentiation}, so that the $\mathbb{Z}_p$-linearity conditions read as follows: for $a,b\in A$ and $\lambda\in \mathbb{Z}_p$, we have $\theta(ab) = \theta(a)+\theta(b)$ and $\theta(a^\lambda) = \lambda \theta(a)$.)

The proof now proceeds as in \cite[\S 2.1(c)]{roseblade-integral-group-rings}. As $\mathbb{Z}_pG$ is polycentral, there exists an element $\eta\in\mathfrak{a}\setminus \mathfrak{ga}$ such that $\eta + \mathfrak{ga}$ is central in $\mathbb{Z}_pG/\mathfrak{ga}$. Choose the element $a\in A$ such that $\theta(a) = \eta+\mathfrak{ga}$: then, for all $g\in G$, we have
$$\theta(gag^{-1}) = (1 - gag^{-1}) + \mathfrak{ga} = g(\eta + \mathfrak{ga})g^{-1} = \eta + \mathfrak{ga} = \theta(a),$$
from which we may conclude that $gag^{-1} = a$.
\end{proof}

\textit{Proof of Theorem E.} This follows from Proposition \ref{propn: polycentral soluble implies nilpotent}, as soluble groups have nontrivial abelian normal subgroups.\qed

The hypothesis that $G$ contain a nontrivial abelian normal subgroup, appearing in Proposition \ref{propn: polycentral soluble implies nilpotent}, does not appear to be necessary in the classical cases of discrete groups or finite-dimensional Lie algebras in characteristic 0. Encouraged by the partial evidence provided by Proposition \ref{propn: polycentral soluble implies nilpotent} and the results of \cite{ardakovInv}, we conjecture:

\begin{conj}
If $G$ is $p$-valuable, then $kG$ is polycentral if and only if $G$ is nilpotent.
\end{conj}

\bibliography{../../biblio/biblio-all}

\begin{thebibliography}{10}

\bibitem{ardakovDocumenta}
K.~Ardakov.
\newblock The centre of completed group algebras of pro-$p$ groups.
\newblock {\em Documenta Math.}, 9:599--606, 2004.

\bibitem{ardakovContemp}
K.~Ardakov.
\newblock The controller subgroup of one-sided ideals in completed group rings.
\newblock {\em Contemporary Mathematics}, 562:11--26, 2012.

\bibitem{ardakovInv}
K.~Ardakov.
\newblock Prime ideals in nilpotent {I}wasawa algebras.
\newblock {\em Inventiones mathematicae}, 190(2):439--503, 2012.

\bibitem{ardakovbrown}
K.~Ardakov and K.A. Brown.
\newblock Primeness, semiprimeness and localisation in {I}wasawa algebras.
\newblock {\em Trans. Amer. Math. Soc.}, 359:1499--1515, 2007.

\bibitem{brumer}
A.~Brumer.
\newblock Pseudocompact algebras, profinite groups and class formations.
\newblock {\em Journal of Algebra}, 4:442--470, 1966.

\bibitem{DDMS}
J.D. Dixon, M.P.F. du~Sautoy, A.~Mann, and D.~Segal.
\newblock {\em Analytic Pro-$p$ Groups}.
\newblock Cambridge University Press, 1999.

\bibitem{GruRos72}
K.W. Gruenberg and J.E. Roseblade.
\newblock The augmentation terminals of certain locally finite groups.
\newblock {\em Can. J. Math.}, XXIV(2):221--238, 1972.

\bibitem{hartshorne}
R.~Hartshorne.
\newblock {\em Algebraic Geometry}.
\newblock Springer, 1977.

\bibitem{LVO}
Li~Huishi and Freddy {van Oystaeyen}.
\newblock {\em Zariskian Filtrations}.
\newblock Springer Science+Business Media, 1996.

\bibitem{jategaonkar}
A.~Jategaonkar.
\newblock {\em Localization in Noetherian rings}.
\newblock Cambridge University Press, 1986.

\bibitem{jones-thesis}
Adam Jones.
\newblock {\em Prime ideals of Iwasawa algebras over Solvable Groups}.
\newblock PhD thesis, University of Oxford, 2020.

\bibitem{jones-woods-1}
Adam Jones and William Woods.
\newblock Skew power series rings over a prime base ring (preprint).
\newblock \texttt{https://arxiv.org/abs/2112.10242}.

\bibitem{lazard}
Michel Lazard.
\newblock Groupes analytiques $p$-adiques.
\newblock {\em Publications Math\'ematiques de l'IH\'ES}, 26:5--219, 1965.

\bibitem{LerMat92}
Andr\'e Leroy and Jerzy Matczuk.
\newblock The extended centroid and {X}-inner automorphisms of {O}re
  extensions.
\newblock {\em Journal of Algebra}, 145:143--177, 1992.

\bibitem{letzter-noeth-skew}
Edward~S. Letzter.
\newblock Prime ideals of noetherian skew power series rings.
\newblock {\em Israel J. Math.}, 192:67--81, 2012.

\bibitem{MR}
J.C. McConnell and J.C. Robson.
\newblock {\em Noncommutative Noetherian Rings}.
\newblock American Mathematical Society, 2001.

\bibitem{passmanICP}
Donald~S. Passman.
\newblock {\em Infinite Crossed Products}.
\newblock Academic Press Inc., 1989.

\bibitem{roseblade-integral-group-rings}
J.E. Roseblade.
\newblock The integral group rings of hypercentral groups.
\newblock {\em Bull. London Math. Soc.}, 3:351--355, 1971.

\bibitem{roseblade}
J.E. Roseblade.
\newblock Prime ideals in group rings of polycyclic groups.
\newblock {\em Proc. London Math. Soc.}, 36:385--447, 1978.

\bibitem{roseblade-smith-hypercentral}
J.E. Roseblade and P.F. Smith.
\newblock A note on hypercentral group rings.
\newblock {\em J. London Math. Soc.}, 13(2):183--190, 1976.

\bibitem{schneider-venjakob-codim}
Peter Schneider and Otmar Venjakob.
\newblock On the codimension of modules over skew power series rings with
  applications to {I}wasawa algebras.
\newblock {\em J. Pure Appl. Algebra}, 204:349--367, 2005.

\bibitem{warner}
Seth Warner.
\newblock {\em Topological Rings}.
\newblock 178. North-Holland Mathematics Studies, 1993.

\bibitem{wilson}
J.S. Wilson.
\newblock {\em Profinite groups}.
\newblock Clarendon Press, Oxford, UK, 1998.

\bibitem{woods-extensions-of-primes}
William Woods.
\newblock Extensions of almost faithful prime ideals in virtually nilpotent
  mod-$p$ {I}wasawa algebras.
\newblock {\em Pacific J. Math.}, 297(2):477--509, 2018.

\bibitem{woods-struct-of-G}
William Woods.
\newblock On the structure of virtually nilpotent compact $p$-adic analytic
  groups.
\newblock {\em J. Group Theory}, 21(1):165--188, 2018.

\bibitem{woods-catenary}
William Woods.
\newblock On the catenarity of virtually nilpotent mod-$p$ {I}wasawa algebras.
\newblock {\em Israel J. Math.}, 238:501--536, 2020.

\bibitem{woods-prime-quotients}
William Woods.
\newblock Maximal prime homomorphic images of mod-$p$ {I}wasawa algebras.
\newblock {\em Math. Proc. Cambridge Philos. Soc.}, 171(2):387--419, 2021.

\bibitem{woods-SPS-dim}
William Woods.
\newblock Dimension theory in iterated local skew power series rings.
\newblock \emph{Algebr. Represent. Theory}. Published online:
  \texttt{https://doi.org/10.1007/s10468-022-10144-3}, 2022.

\end{thebibliography}
\bibliographystyle{plain}

\end{document}